\documentclass{amsart}
\usepackage{amssymb}
\usepackage{pdflscape}
\usepackage{amscd}
\usepackage{fancyhdr}
\usepackage{graphicx}

\usepackage[usenames,dvipsnames]{xcolor} 
\usepackage{tikz, tikz-3dplot, pgfplots}
\usepackage{tkz-graph}
\usepackage{tikz-cd}
\usetikzlibrary{positioning, calc, arrows, decorations.markings, decorations.pathreplacing}
\tikzset{
  symbol/.style={
    draw=none,
    every to/.append style={
      edge node={node [sloped, allow upside down, auto=false]{$#1$}}}
  }
}
\tikzcdset{scale cd/.style={every label/.append style={scale=#1},
    cells={nodes={scale=#1}}}}

\newtheorem{theorem}{Theorem}[section]
\newtheorem{corollary}[theorem]{Corollary}
\newtheorem{lemma}[theorem]{Lemma}

\newtheorem{definition}[theorem]{Definition}

\begin{document}
\title{Floer potentials, cluster algebras and quiver representations}
\author{Peter Albers, Maria Bertozzi, Markus Reineke}
\address{Peter Albers:\newline Universität Heidelberg, Institut für Mathematik, Im Neuenheimer Feld 205, 69120 Heidelberg, Germany}
\email{palbers@mathi.uni-heidelberg.de}
\address{Maria Bertozzi:\newline Fakultät für Mathematik, Ruhr-Universität Bochum, Universitätsstraße 150, 44780 Bochum, Germany}
\email{maria.bertozzi@rub.de}
\address{Markus Reineke:\newline Fakultät für Mathematik, Ruhr-Universität Bochum, Universitätsstraße 150, 44780 Bochum, Germany}
\email{markus.reineke@rub.de}

\begin{abstract} We use cluster algebras to interpret Floer potentials of monotone Lagrangian tori in toric del Pezzo surfaces as cluster characters of quiver representations (see Theorem \ref{main} for the precise statement). 
\end{abstract}
\maketitle
\parindent0pt

\section{Introduction}

In the seminal work \cite{V}, R.~Vianna constructed infinitely many pairwise non Hamiltonian isotopic monotone Lagrangian tori in the complex projective plane. These tori are naturally indexed by Markov triples (more precisely, by the triple of squares), and were constructed inductively in \cite{V2} by a geometric mutation procedure along the Markov tree. To distinguish these tori, Vianna uses information on counts of Maslov index $2$ holomorphic disks contained in them. \\[1ex]
This construction was further conceptualized in \cite{PT}, where it is put into a context of cluster-like mutation operations originating in \cite{CMG}. Vianna's mutation operation is extended to mutations of so-called Lagrangian seeds in del Pezzo surfaces (similar concepts are the arboreal Lagrangian skeleta of \cite{N} and the compressing systems of \cite{CW}), the main result being a wall-crossing formula, stating that the Floer potentials of mutated Lagrangian tori are related by an explicit algebraic mutation rule for Laurent polynomials (see Sections \ref{lgp}, \ref{fp} for a summary).\\[1ex]
A formally similar mutation invariance is known in the representation-theoretic approach to cluster algebras \cite{DWZ2}, to be recalled in Sections \ref{ca}, \ref{rqp}. Namely, the $F$-polynomials of representations of quivers with potentials (or more precisely the cluster characters), encoding Euler characteristics of Grassmannians of subrepresentations, are related by the mutation rule in cluster algebras when the representations are mutated.\\[1ex]
This formal similarity, and experiments for the first three Vianna tori, led the authors to the prediction that, to any of the monotone Lagrangian tori in toric del Pezzo surfaces constructed in \cite{PT}, one can associate a representation of a quiver with potential, whose cluster character corresponds naturally to the Floer potential of the torus.\\[1ex]
Indeed, this turns out to be feasible. Inspired by the setup of \cite{GHK}, we construct (see Section \ref{comparison}) a comparison map from two-variable Laurent polynomials to cluster algebras, which is compatible with the two algebraic mutation rules mentioned above. We state and prove an appropriate version of mutation invariance of cluster characters in Section \ref{rqp}. \\[1ex]
After realizing the ``initial'' Landau-Ginzburg seeds of \cite{PT} as cluster characters in Section \ref{fpfp}, the compatibilities of the various mutations then almost directly imply the main result Theorem \ref{main}, stating that indeed all relevant Floer potentials are cluster characters of naturally constructed representations.\\[2ex]
{\bf Acknowledgments:} This work was initiated at the workshop ``Interplay betweeen symplectic geometry and cluster theory'' at Heidelberg in January 2023. The authors would like to thank all organizers, speakers and participants for the inspiring atmosphere. M. R. would like to thank B.~de Laporte, D.~Kaufman, B.~Keller,~D. Labardini-Fragoso and P.~Plamondon for many clarifying discussions on mutation and cluster characters. Thanks to an anonymous referee for suggesting several improvements to the presentation. The work of all authors is supported by the DFG CRC-TRR 191 ``Symplectic structures in geometry, algebra and dynamics'' (281071066).

\section{Recollections}
\subsection{Landau-Ginzburg potentials}\label{lgp}
Here we follow closely the presentation in \cite[Section 4]{PT}. For real numbers $a$ we define $[a]_+=\max(a,0)$, and make frequent use of the identity $$a=[a]_+-[-a]_+.$$ We consider integral primitive vectors $v\in\mathbb{Z}^2$. For $v=(a,b)$ we write $$v^\perp=(b,-a).$$ Given another such vector $w=(c,d)$, we consider the standard scalar product $$(v,w)=ac+bd,$$ as well as the symplectic form $$\{v,w\}=ad-bc$$ and note that $$\{v,w\}=-(v^\perp,w).$$ The vector $v=(a,b)$ defines a mutation map $\mu_v$ on rational functions in $\mathbb{C}(z_1,z_2)$ by
$$\mu_v(F)(z_1,z_2)=F(z_1(1+z_1^bz_2^{-a})^{-a},z_2(1+z_1^bz_2^{-a})^{-b})$$
(this is the cluster $\mathcal{X}$-mutation in the sense of \cite{GHK}). 
On monomials $z^w=z_1^cz_2^d$ we thus have
$$\mu_v(z^w)=z^w(1+z^{v^\perp})^{-(v,w)}.$$ As a tropical version, we define a mutation map on vectors by 
$$\mu_vw=w+[\{w,v\}]_+v.$$

\begin{definition}\label{deflgseed} A {\it seed} is a collection $${\bf s}=(W,v_1,\ldots,v_n)$$ consisting of a Laurent polynomial $W\in\mathbb{C}[z_1^\pm,z_2^\pm]$, called a {\it potential}, and pairwise different primitive vectors $v_1,\ldots,v_n\in\mathbb{Z}^2$, called {\it directions}. The seed ${\bf s}$ is called a {\it Landau-Ginzburg (LG) seed} if $\mu_{v_i}W$ is again a Laurent polynomial for all $i=1,\ldots,n$. In this case, we define $\mu_i^S{\bf s}=\mu_i{\bf s}$, the seed mutation of ${\bf s}$ in the direction $i=1,\ldots,n$, as the seed $$(\mu_{v_i}W,v_1',\ldots,v_n'),$$
where
$$v_i'=-v_i,\; v_j'=\mu_{v_i}v_j,\, j\not=i.$$
\end{definition}

Note that the operation of mutation of LG seeds is not involutive, since $$\mu_{-v}\mu_vw=w+\{w,v\}v$$ defines a transvection, and similarly
$$\mu_{-v}(\mu_v(z^w))=\mu_{-v}(z^w(1+z^{v^\perp})^{-(v,w)})=$$
$$=z^w(1+z^{-v^\perp})^{(v,w)}(1+z^{v^\perp})^{-(v,w)}=z^{w-(v^\perp,w)v^\perp}=z^{w+\{v,w\}v^\perp}.$$

\begin{theorem}\cite[Corollary 3.1]{CMG} Any mutation of an LG seed is again an LG seed.
\end{theorem}

In other words, an LG seed stays an LG seed under arbitrary sequences of mutations.\\[1ex]
To see the concrete examples of LG seeds we will be interested in, we reproduce the list of LG seeds corresponding to toric del Pezzo surfaces from \cite{PT}.

\begin{theorem}\cite[Table 1]{PT}\label{delpezzoseeds} Any of the following combinations of toric del Pezzo surfaces, potentials and directions constitutes an LG seed:
\begin{center}
\begin{tabular}{|c|l|}\hline
del Pezzo & potential, directions\\ \hline\hline
$\mathbb{C}P^2$&  $z_1+z_2+z_1^{-1}z_2^{-1}$,\\ \hline
&  $(1,1),(-2,1),(1,-2)$\\ \hline
$\mathbb{C}P^1\times\mathbb{C}P^1$&  $z_1+z_2+z_1^{-1}+z_2^{-1}$,\\ \hline  & $(1,1),(1,-1),(-1,1),(-1,-1)$\\ \hline
${\rm Bl}_1\mathbb{C}P^2$& $z_1+z_2+z_1^{-1}z_2^{-1}+z_1z_2$, \\ \hline &  $(-2,1),(1,-2),(1,0),(0,1)$\\  \hline
${\rm Bl}_2\mathbb{C}P^2$&  $z_1+z_2+z_1^{-1}+z_2^{-1}+z_1^{- 1}z_2^{-1}$,\\  \hline &  $(1,-1),(-1,1),(-1,0),(0,-1),(1,1)$\\ \hline
${\rm Bl}_3\mathbb{C}P^2$& $z_1+z_2+z_1^{-1}+z_2^{-1}+z_1z_2+z_1^{-1}z_2^{-1}$, \\ \hline &  $\pm(1,-1),\pm(1,0),\pm(0,1)$\\ \hline
\end{tabular}
\end{center}

\end{theorem}

In the following, we take the above LG seeds as initial ones, and consider all of their mutations:

\begin{definition}\label{alllgseeds} Let $X$ be a toric del Pezzo surface, and let $${\bf s}(X)=(W(X),v_1,\ldots,v_n),$$ be the corresponding LG seed in the above table. For ${\bf i}=(i_1,\ldots,i_N)$ a sequence of indices in $\{1,\ldots,n\}$ without repetitions, define $$s_{\bf i}(X)=\mu_{i_N}\cdots\mu_{i_1}{\bf s}(X)$$ as the corresponding mutation of ${\bf s}(X)$, and denote by $$W_{\bf i}(X)=\mu_{i_N}\cdots\mu_{i_1}W(X)$$ its potential.
\end{definition}

\subsection{Floer potentials}\label{fp}

To define Floer potentials and their mutation, we continue to follow \cite{PT}. We continue to denote by $X$ a toric del Pezzo surface viewed as symplectic manifold. For a monotone Lagrangian submanifold $L\subset X$, we parametrize ${\rm GL}_1(\mathbb{C})$-local systems by ${\rm Hom}(H_1(L,\mathbb{Z}),\mathbb{C}^*)$. We then define the Floer potential (equivalently, disk or Landau-Ginzburg potential) $$W_L:{\rm Hom}(H_1(L,\mathbb{Z}),\mathbb{C}^*)\rightarrow\mathbb{C}$$
of $L$ as the count of Maslov index $2$ holomorphic disks in $X$ with boundary on $L$, weighted by the holonomy of the local system on $L$:
$$W_L(\rho)=\sum_\delta n_\delta\rho(\partial\delta),$$
where the sum ranges over classes $\delta\in H_2(X,L)$ of Maslov index $2$, the number $n_\delta$ is the count of holomorphic disks in class $\delta$, and $\rho(\partial\delta)$ is the holonomy along the boundary loop of the disk of the local system determined by $\rho$. More precisely, $n_\delta$ is the degree of the map ``evaluation at the marked point'' $$\mathcal{M}(L,\delta)\rightarrow L$$ on the moduli space of holomorphic disks with boundary on $L$ of Maslov index $2$ with a single marked point on the boundary of the disk representing the class $\delta$ (with respect to a generic almost complex structure). Identifying $H_1(L,\mathbb{Z})$ with $\mathbb{Z}^2$ 
%/\mbox{torsion}\simeq\mathbb{Z}^2$$
by chosing a basis, $\partial \delta$ becomes an element $(l_1,l_2)$ of $\mathbb{Z}^2$. Monotonicity of $L$ and Gromov compactness then imply that $W_L$, viewed as a function $(\mathbb{C}^*)^2\rightarrow\mathbb{C}$ (using the previous identifications), is a finite sum of terms $n_\delta{\bf z}^{\partial\delta}$ (where ${\bf z}^{\partial\delta}=z_1^{l_1}z_2^{l_2}$) over the classes $\delta$ of Maslov index $2$, and can thus be viewed as a Laurent polynomial in $\mathbb{C}[z_1^\pm,z_2^\pm]$,
$$W_L=\sum_\delta n_\delta\mathbf{z}^{\partial\delta}.$$

For a monotone Lagrange torus $L\subset X$ and a Lagrangian disk $D\subset X$ with boundary on $L$ fulfilling suitable transversality conditions, there is an operation of mutation of $L$ in direction $D$, producing a mutated Lagrangian torus $\mu_DL$ and a mutated  Lagrangian disk $\mu_DD$. This operation is involutive up to a Hamiltonian isotopy.\\[1ex]
A Lagrangian seed $(L,\{D_i\}_i)$ in $X$ consists of a monotone Lagrangian torus $L$ and embedded Lagrangian disks fulfilling suitable transversality conditions. The above mutation procedure can be extended to provide mutations
$$\mu_i^L(L,\{D_j\}_j)=\mu_i(L,\{D_j\}_j)=(\mu_{D_i}L,\{\mu_{D_i}D_j\}_j)$$
of Lagrangian seeds in arbitrary directions.\\[1ex]
Using the isomorphism $H_1(L,\mathbb{Z})\simeq\mathbb{Z}^2$, we assign to a Lagrangian seed $(L,\{D_i\}_i)$ the LG seed $$W(L,\{D_i\}_i)=(W_L,([\partial D_i]^\perp)_i).$$

\begin{theorem}\cite[Theorem 4.8, Proposition 4.22]{PT}\label{mfp} The map $W$ assigning an LG seed to a Lagrangian seed is compatible with the mutation operations, that is,
$$\mu_i^S\circ W=W\circ \mu_i^L.$$
 Moreover, there exists an {\rm initial} Lagrangian seed $(L(X),\{D_i(X)\}_i)$ realizing the LG seeds ${\bf s}(X)$ of Theorem \ref{delpezzoseeds}.
\end{theorem}

We can then define iterated mutated Lagrangian seeds
$$(L_{\bf i}(X),\{D_j^{\bf i}(X)\}_j)=\mu_{i_N}\cdots\mu_{i_1}(L(X),\{D_i(X)\}_i)$$
for any sequence ${\bf i}=(i_1,\ldots,i_N)$ without repetitions. By the previous theorem, we thus have

\begin{corollary} For all sequences ${\bf i}$ as above, we have an equality of LG seeds
$$(W_{L_{\bf i}(X)},\{[\partial D_j^{\bf i}(X)]\}_j)={\bf s}_{\bf i}(X).$$
\end{corollary}

\subsection{Cluster algebras}\label{ca}
We now turn to the purely algebraic part and recall some cluster algebra terminology following \cite{FZ4}. Let $B=(b_{i,j})$ be a skew-symmetric $n\times n$-matrix. We define the mutation $\mu_iB=B'=(b_{j,k}')$ in direction $i=1,\ldots,n$ of $B$ by
$$b_{j,k}'=b_{j,k}\mbox{ if } j=i\mbox{ or }k=i,$$
$$b_{j,k}'=b_{j,k}+[b_{j,i}]_+[b_{i,k}]_+-[-b_{j,i}]_+[-b_{i,k}]_+\mbox{ if } j\not=i\not=k.$$ We define the mutation of $x$-variables as the automorphism $\mu_i^C=\mu_i$ of the rational function field $\mathbb{C}(x_1,\ldots,x_n)$ given by $\mu_i(x_j)=x_j'$ for
$$x_i'={x_i}^{-1}\prod_j{x_j}^{[b_{i,j}]_+}(1+\prod_j{x_j}^{-b_{i,j}})\mbox{ and } x_j'=x_j,\, j\not=i,$$
a variant of the more common $$x_i'=x_i^{-1}(\prod_jx_j^{[b_{i,j}]_+}+\prod_jx_j^{[b_{j,i}]_+}).$$

We define the $y$-variables by $$y_i=\prod_{j}x_j^{b_{j,i}},$$
and similarly for $y_i'$, so that the previous mutation rule reads
$$x_i'=x_i^{-1}\prod_j{x_j}^{[b_{i,j}]_+}(1+{y_i}).$$
The mutation of $y$-variables is then given by
$$y_i'={y_i}^{-1}\mbox{ and } y_j'={y_j}{y_i}^{[b_{i,j}]_+}(1+{y_i})^{-b_{i,j}},\, j\not=i.$$

\subsection{Representations of quivers with potentials}\label{rqp}

Turning to the quiver theoretic approach to cluster algebras and mutation, we introduce representations of quivers with potential following \cite{DWZ2}. Given a matrix $B$ as in the previous subsection, we associate to it a quiver $Q=Q(B)$ with vertices $i=1,\ldots,n$, such that $$\mbox{ for all $i,j\in Q_0$, there are }[-b_{i,j}]_+\mbox{ arrows from $i$ to $j$}.$$
Note that $Q(B)$ has, by definition, no loops or two-cycles. We have
$$Q(\mu_iB)=\mu_iQ(B),$$
where the mutation $\mu_iQ$ of a quiver at a vertex $i$ is defined by reversing all arrows incident with $i$, adding an arrow $$[\beta\alpha]:j\rightarrow k$$ for every pair of arrows $$j\stackrel{\alpha}{\rightarrow}i\stackrel{\beta}{\rightarrow k},$$
and deleting all resulting two-cycles.\\[1ex]
We denote by $\widehat{\mathbb{C}Q}$ the completed (with respect to path length) path algebra of $Q$%, and by $\mathfrak{m}$ its maximal ideal
. The factor $\widehat{\mathbb{C}Q}/[\widehat{\mathbb{C}Q},\widehat{\mathbb{C}Q}]$ by (additive) commutators is spanned by cyclic equivalence classes of formal sums of oriented cycles. An element $S$ in this space is called a potential for $Q$, for which cyclic derivatives $\partial_\alpha S$ along the arrows $\alpha$ in $Q$ are defined, as well as double cyclic derivatives $\partial_{\beta\alpha} S$. For the detailed definition, we refer to \cite[(4.3),(4.4)]{DWZ2}.\\[1ex]
For a sufficiently general (non-degenerate in the terminology of \cite{DWZ2}) potential $S$ on $Q$,
one can naturally define a mutation $\mu_iS$ as a potential on $\mu_iQ$ (note that such potentials always exist by \cite[Corollary 7.4]{DWZ1}).\\[1ex]
We consider representations $V$ of $Q$ which satisfy the Jacobian relations that all $\partial_\alpha S$ are represented by zero in $V$, thus
$$V_{\partial_\alpha S}=0,$$ which we call representations of the quiver with potential $(Q,S)$. Moreover, we require all such representations to be nilpotent, that is, every oriented cycle $\omega$ is represented by a nilpotent endomorphism $V_\omega$.\\[1ex]
Slightly more generally, we consider decorated representations $(V,V^-)$, consisting of a representation $V$ of $(Q,S)$, together with vector spaces $V^-_i$ for $i=1,\ldots,n$. As a special case, we consider the negative simple decorated representation $S_i^-$: it consists of the zero quiver representation, and the additional vector spaces are given by  $(S_i^-)_j^-=0$ for $j\not=i$ and $(S_i^-)^-_i=\mathbb{C}$. \\[1ex]
For every (decorated) representation $V$ of $(Q,S)$ (the potential $S$ still assumed to be non-degenerate) and all $i\in Q_0$, a mutation $\mu_iV$ is defined in \cite{DWZ2}, which is a (decorated) representation of $(\mu_iQ,\mu_iS)$. The operation of mutation is involutive up to so-called right equivalence, that is, $\mu_i\mu_iV$ is isomorphic to $\varphi^*V$ for some automorphism $\varphi$ of $\widehat{\mathbb{C}Q}$ fixing $S$.

\subsection{$F$-polynomials and cluster characters} For a decorated representation $V$ of a quiver $Q$ and a dimension vector ${\bf e}$ with entries $e_i\leq\dim V_i$ for all $i$ we define the quiver Grassmannian ${\rm Gr}_{\bf e}(V)$ as the complex projective variety of ${\bf e}$-dimensional subrepresentations of $V$, viewed for example as a Zariski-closed subset of the product $\prod_i{\rm Gr}_{e_i}(V_i)$ of ordinary  Grassmannians. We define the $F$-polynomial of $V$ encoding the topological Euler characteristic $\chi({\rm Gr}_{\bf e}(V))$ of all quiver Grassmannians:
$$F_V(u)=\sum_{\bf e}\chi({\rm Gr}_{\bf e}(V))\prod_iu_i^{e_i}\in\mathbb{C}[u_1,\ldots,u_n].$$

We define the ${\bf g}$- and ${\bf h}$-vectors of $V$:\\[1ex]
Fixing a vertex $i$ of $Q$, any decorated representation $V$ gives rise to a triangle of linear maps

\begin{center}
\tikzcdset{row sep/normal=1cm}
\begin{tikzcd}
 & V_i\arrow[dr, "\beta_i"] & \\
 V_i^{\rm in} \arrow[ur, "\alpha_i"] &  & V_i^{\rm out} \arrow[ll, "\gamma_i"]
\end{tikzcd}
\end{center}

where
$$V_i^{\rm in}=\bigoplus_{\alpha:j\rightarrow i}V_j,\; V_i^{\rm out}=\bigoplus_{\beta:i\rightarrow j}V_j,$$
$$\alpha_i=(V_\alpha)_{\alpha:j\rightarrow i},\; \beta_i=(V_\beta)_{\beta:i\rightarrow j},$$
and the $(\alpha,\beta)$-component of $\gamma_i$ is given by
$$(\gamma_i)_{(\alpha,\beta)}=V_{\partial_{\beta\alpha}S}.$$

We then define $$h_i(V)=-\dim{\rm Ker}(\beta_i)$$ and $$g_i(V)=\dim{\rm Ker}(\gamma_i)-\dim V_i+\dim V_i^-,$$ 
the latter using the decoration datum $V^-$. We have the following relation between ${\bf g}$- and ${\bf h}$-vectors: $$g_i(V)=h_i(V)-h_i(\mu_iV).$$ The ${\bf g}$-vectors fulfill the mutation rule
$$g_i(\mu_iV)=-g_i(V),\; g_j(\mu_iV)=g_j(V)+[b_{j,i}]_+g_i(V)-b_{j,i}h_i(V),\, i\not=j.$$

We have the following mutation rule for $F$-polynomials, using the notation for mutation of variables of Section \ref{ca}:

\begin{lemma}\cite[Lemma 5.2]{DWZ2} We have
$$(1+y_i)^{h_i(V)}F_V(y)=(1+{y_i'})^{h_i(\mu_i(V))}F_{\mu_iV}(y').$$
\end{lemma}

The following is the central definition for our purposes, originating in \cite{CC}, and written in this form in \cite[Remark 5.2]{P}:

\begin{definition} We define the cluster character of a decorated representation $V$ by
$${\rm CC}_V(x)=x^{{\bf g}(V)}F_V(y)\in\mathbb{C}[x_1^\pm,\ldots,x_n^\pm].$$
\end{definition}

For two decorated representations $V$ and $W$ of $Q$, we have
$$F_{V\oplus W}(x)=F_V(x)\cdot F_W(x)$$ and
$${\bf g}(V\oplus W)={\bf g}(V)+{\bf g}(W),$$
so that
$${\rm CC}_{V\oplus W}(x)={\rm CC}_V(x)\cdot {\rm CC}_W(x).$$
It is thus natural to extend cluster characters  to virtual representations, that is, elements of the split Grothendieck group $K_0^\oplus$ of the category of decorated representations. Namely, every element of this group can be written as a difference $[V]-[W]$ of isomorphism classes of representations, and we define
$${\rm CC}_{[V]-[W]}(x)=\frac{{\rm CC}_V(x)}{{\rm CC}_W(x)}\in\mathbb{C}(x_1,\ldots,x_n).$$
Note also that the mutation operation $\mu_i$ naturally extends to virtual representations.

\begin{lemma}\label{wk} The cluster character of virtual representations is invariant under mutation, that is,
$${\rm CC}_V(x)={\rm CC}_{\mu_iV}(x').$$
\end{lemma}

\begin{proof} This lemma is essentially \cite[Corollary 4.14]{P0}. Due to different conventions in notation, we give a complete proof here. We can assume $V$ to be a decorated representation. Using the mutation rule for $F$-polynomials, the relation between ${\bf g}$- and ${\bf h}$-vectors and the definition of $y$-variables, we have $$x^{{\bf g}(V)}F_V(y)=x^{{\bf g}(V)}(1+y_i)^{-h_i(V)}(1+y_i)^{h_i(V)}F_V(y)=$$
$$=x^{{\bf g}(V)}(1+y_i)^{-h_i(V)}(1+y_i^{-1})^{h_i(\mu_iV)}F_{\mu_iV}(y')=$$
$$=x^{{\bf g}(V)}(1+y_i)^{-h_i(V)+h_i(\mu_iV)}y_i^{-h_i(\mu_iV)}F_{\mu_iV}(y')=$$
$$=x^{{\bf g}(V)}(1+y_i)^{-g_i(V)}\prod_{j\not=i}x_j^{-b_{j,i}h_i(\mu_iV)}F_{\mu_iV}(y')=$$
$$=x_i^{g_i(V)}\prod_{j\not=i}x_j^{g_j(V)-b_{j,i}h_i(\mu_iV)}(1+y_i)^{-g_i(V)}F_{\mu_iV}(y').$$
We rewrite the exponent of $x_j$ using the mutation rule for ${\bf g}$-vectors as
$$g_j(V)-b_{j,i}h_i(\mu_iV)=g_j(V)-b_{j,i}h_i(V)+b_{j,i}g_i(V)=$$
$$=g_j(\mu_iV)-[b_{j,i}]_+g_i(V)+b_{j,i}g_i(V)=$$
$$=g_j(\mu_iV)-[b_{i,j}]_+g_i(V)$$
and thus $x^{{\bf g}(V)}F_V(y)=$
$$=x_i^{g_i(V)}\prod_{j\not=i}x_j^{g_j(\mu_iV)-[b_{i,j}]_+g_i(V)}(1+y_i)^{-g_i(V)}F_{\mu_iV}(y')=$$
$$=(x_i^{-1}\prod_{j\not=i}x_j^{[b_{i,j}]_+}(1+y_i))^{g_i(\mu_iV)}\prod_{j\not=i}x_j^{g_j(\mu_iV)}F_{\mu_iV}(y')=$$
$$=(x')^{{\bf g}(\mu_iV)}F_{\mu_iV}(y'),$$
proving the lemma.\end{proof}

\section{LG seeds and cluster algebras}\label{comparison}

To an LG seed ${\bf s}=(W,v_1,\ldots,v_n)$ we associate the skew-symmetric matrix $B=B({\bf s})=(b_{i,j})$ given by
$$b_{i,j}=\{v_i,v_j\}.$$ Note that it has rank at most two. This definition is compatible with mutation, that is, $$B(\mu_i({\bf s}))=\mu_i(B({\bf s})).$$

Namely, we have, in the notation of Definition \ref{deflgseed}:
$$\{v_i',v_k'\}=\{-v_i,v_k+[\{v_k,v_i\}]_+v_i\}=-\{v_i,v_k\},$$
for $i\not=k$, and similarly for $\{v_j',v_i'\}$ and $j\not=i$, whereas, for $j\not=i\not=k$, we have
$$\{v_j',v_k'\}=\{v_j+[\{v_j,v_i\}]_+v_i,v_k+[\{v_k,v_i\}]_+v_i\}=$$
$$=\{v_j,v_k\}+[\{v_k,v_i\}]_+\{v_j,v_i\}+[\{v_j,v_i\}]_+\{v_i,v_k\}=$$
$$=b_{j,k}+[b_{j,i}]_+
b_{i,k}+b_{j,i}[b_{k,i}]_+=
b_{j,k}+[b_{j,i}]_+[b_{i,k}]_+
-[-b_{j,i}]_+[-b_{i,k}]_+.$$

Inspired by the treatment of cluster algebras in \cite[Section 2]{GHK}, we define

\begin{definition} The {\it comparison map} $$\Phi=\Phi_{\bf s}:\mathbb{C}[z_1^\pm,z_2^\pm]\rightarrow\mathbb{C}[x_1^\pm,\ldots,x_n^\pm]$$ is defined by $$\Phi_{\bf s}(z^v)=\prod_ix_i^{-(v,v_i)}.$$
\end{definition}

\begin{lemma}\label{phimap} The comparison map $\Phi_{\bf s}$ is compatible with the mutations of Sections \ref{lgp} and \ref{ca}, respectively, that is, we have $$\mu_i^C\circ\Phi_{\bf s}=\Phi_{\mu_i{\bf s}}\circ\mu_i^S.$$
\end{lemma}

\begin{proof} Namely, we have $$\mu_i(\Phi_{\bf s}(z^v))=\mu_i(\prod_jx_j^{-(v,v_j)})=$$
$$=({x_i}^{-1}\prod_j{x_j}^{[b_{i,j}]_+}(1+y_i))^{-(v,v_i)}\prod_{j\not=i}{x_j}^{-(v,v_j)}=$$
$$={x_i}^{(v,v_i)}\prod_{j\not=i}{x_j}^{-(v,v_j)-[b_{i,j}]_+(v,v_i)}(1+y_i)^{-(v,v_i)}.$$
On the other hand, we have
$$\Phi_{\mu_i{\bf s}}(\mu_i(z^v))=\Phi_{\mu_i{\bf s}}(z^v(1+z^{v_i^\perp})^{-(v_i,v)})=$$
$$=\prod_j{x_j}^{-(v,v_j')}(1+\prod_j{x_j}^{-(v_i^\perp,v_j')})^{-(v_i,v)}.$$
Now
$$\prod_j{x_j}^{-(v_i^\perp,v_j')}=\prod_j{x_j}^{\{v_i,v_j'\}}=\prod_j{x_j}^{b_{i,j}}={y_i^{-1}},$$
so that
$$(1+\prod_j{x_j}^{-(v_i^\perp,v_j')})^{-(v_i,v)}=(1+{y_i}^{-1})^{-(v_i,v)}={y_i}^{(v_i,v)}(1+{y_i})^{-(v_i,v)},$$
and
$$\prod_j{x_j}^{-(v,v_j')}={x_i}^{(v,v_i)}\prod_{j\not=i}{x_j}^{-(v,v_j)-[b_{j,i}]_+(v,v_i)},$$
which combines to $\Phi_{\mu_i{\bf s}}(\mu_i(z^v))=$
$$={x_i}^{(v,v_i)}\prod_{j\not=i}{x_j}^{-(v,v_j)-[b_{j,i}]_+(v,v_i)}{y_i}^{(v_i,v)}(1+{y_i})^{-(v_i,v)}=$$
$$={x_i}^{(v,v_i)}\prod_{j\not=i}{x_j}^{-(v,v_j)-[b_{j,i}]_+(v,v_i)-b_{i,j}(v_i,v)}(1+{y_i})^{-(v_i,v)}=$$
$$={x_i}^{(v,v_i)}\prod_{j\not=i}{x_j}^{-(v,v_j)-[b_{i,j}]_+(v,v_i)}(1+{y_i})^{-(v_i,v)}=$$
$$=\mu_i(\Phi_{\bf s}(z^v))$$
as claimed.\end{proof}

\section{Floer potentials as $F$-polynomials}\label{fpfp}

For each of the LG seeds of Theorem \ref{delpezzoseeds}, we will now describe first the $B$-matrix induced by the seed, and then the corresponding quiver:

$$\mathbb{C}P^2:\left[\begin{array}{rrr}0&3&-3\\ -3&0&3\\ 3&-3&0\end{array}\right],\;
\mathbb{C}P^1\times\mathbb{C}P^1:\left[\begin{array}{rrrr}0&-2&2&0\\ 2&0&0&-2\\ -2&0&0&2\\ 0&2&-2&0\end{array}\right],$$
$${\rm Bl}_1\mathbb{C}P^2:\left[\begin{array}{rrrr}0&3&-1&-2\\ -3&0&2&1\\ 1&-2&0&1\\ 2&-1&-1&0\end{array}\right],\;
{\rm Bl}_2\mathbb{C}P^2:\left[\begin{array}{rrrrr}0&0&-1&-1&2\\ 0&0&1&1&-2\\ 1&-1&0&1&-1\\ 1&-1&-1&0&1\\ -2&2&1&-1&0\end{array}\right],$$
$${\rm Bl}_3\mathbb{C}P^2:\left[\begin{array}{rrrrrr}0&0&1&-1&1&-1\\ 0&0&-1&1&-1&1\\ -1&1&0&0&1&-1\\ 1&-1&0&0&-1&1\\ -1&1&-1&1&0&0\\ 1&-1&1&-1&0&0\end{array}\right].$$

$\mathbb{C}P^2:$\begin{center}
\begin{tikzpicture}[scale=2.0]
\draw[arrows={-angle 45}, shorten >=6.5, shorten <=6.5]   (0.0, 0.0) -- (1.5, 0.0);
\draw[arrows={-angle 45}, shorten >=6.5, shorten <=6.5]  (0,0.0) to [out=20,in=160]  (1.5,0.0);
\draw[arrows={-angle 45}, shorten >=6.5, shorten <=6.5]  (0,0) to [out=340, in=200] (1.5,0);
	
\draw[arrows={-angle 45}, shorten >=6.5, shorten <=6.5]  (1.5, 0) -- (0.75, 1);
\draw[arrows={-angle 45}, shorten >=6.5, shorten <=6.5]  (1.5,0) to [out=150, in=280] (0.75,1);
\draw[arrows={-angle 45}, shorten >=6.5, shorten <=6.5]  (1.5,0) to [out=100, in=330] (0.76,1);
		
\draw[arrows={-angle 45}, shorten >=6.5, shorten <=6.5]   (0.75, 1) -- (0,0);
\draw[arrows={-angle 45}, shorten >=6.5, shorten <=6.5]  (0.75, 1) to [out=210, in=80] (0,0);
\draw[arrows={-angle 45}, shorten >=6.5, shorten <=6.5]  (0.75, 1) to [out=260, in=30] (0,0);

	\node at (0,0) {$3$};
	\node at (1.5,0) {$2$};
	\node at (0.75,1) {$1$};

\end{tikzpicture}
\end{center}
$\mathbb{C}P^1\times\mathbb{C}P^1$:
\begin{center}
\begin{tikzpicture}[scale=2.0]

\draw[arrows={-angle 45}, shorten >=6, shorten <=6]  (1.5,0) to [out=160, in=280] (0.75,0.8);
\draw[arrows={-angle 45}, shorten >=6, shorten <=6]  (1.5,0) to [out=100, in=340] (0.75,0.8);

\draw[arrows={-angle 45}, shorten >=6, shorten <=6]  (0.75, 0.8) to [out=200, in=80] (0,0);
\draw[arrows={-angle 45}, shorten >=6, shorten <=6]  (0.75, 0.8) to [out=260, in=20] (0,0);

\draw[arrows={-angle 45}, shorten >=6, shorten <=6]  (0,0) to [out=280, in=160] (0.75,-0.8);
\draw[arrows={-angle 45}, shorten >=6, shorten <=6]  (0,0) to [out=340, in=100] (0.75,-0.8);

\draw[arrows={-angle 45}, shorten >=6, shorten <=6]  (0.75, -0.8) to [out=80, in=200] (1.5,0);
\draw[arrows={-angle 45}, shorten >=6, shorten <=6]  (0.75, -0.8) to [out=20, in=260] (1.5,0);

	\node at (0,0) {$2$};
	\node at (1.5,0) {$3$};
	\node at (0.75,0.8) {$1$};
	\node at (0.75, -0.8) {$4$};

\end{tikzpicture}
\end{center}
${\rm Bl}_1\mathbb{C}P^2$:
\begin{center}
\begin{tikzpicture}[scale=2.0]

\draw[arrows={-angle 45}, shorten >=6, shorten <=6]  (1.5,0) to [out=160, in=280] (0.75,0.8);
\draw[arrows={-angle 45}, shorten >=6, shorten <=6]  (1.5,0) to [out=100, in=340] (0.75,0.8);
\draw[arrows={-angle 45}, shorten >=6, shorten <=6]  (1.5, 0) -- (0.75,0.8);

\draw[arrows={-angle 45}, shorten >=6, shorten <=6]  (0.75, 0.8) to [out=210, in=70] (0,0);
\draw[arrows={-angle 45}, shorten >=6, shorten <=6]  (0.75, 0.8) to [out=250, in=30] (0,0);
\draw[arrows={-angle 45}, shorten >=6, shorten <=6]  (0.75, 0.8) -- (0.75,-0.8);

\draw[arrows={-angle 45}, shorten >=6, shorten <=6]  (0,0) -- (0.75,-0.8);
\draw[arrows={-angle 45}, shorten >=6, shorten <=6]  (0,0) -- (1.5,0);

\draw[arrows={-angle 45}, shorten >=6, shorten <=6]  (0.75, -0.8) to [out=70, in=210] (1.5,0);
\draw[arrows={-angle 45}, shorten >=6, shorten <=6]  (0.75, -0.8) to [out=30, in=250] (1.5,0);

	\node at (0,0) {$4$};
	\node at (1.5,0) {$2$};
	\node at (0.75,0.8) {$1$};
	\node at (0.75, -0.8) {$3$};

\end{tikzpicture}
\end{center}
${\rm Bl}_2\mathbb{C}P^2$:
\begin{center}
\begin{tikzpicture}[scale=2.0]
\draw[arrows={-angle 45}, shorten >=6, shorten <=6]  (2,0.4) to [out=170, in=290] (1,1.2);
\draw[arrows={-angle 45}, shorten >=6, shorten <=6]  (2,0.4) to [out=110, in=350] (1,1.2);
\draw[arrows={-angle 45}, shorten >=6, shorten <=6]  (1.5, -0.8) to [out=20, in=280] (2,0.4);
\draw[arrows={-angle 45}, shorten >=6, shorten <=6]  (1.5, -0.8) -- (2,0.4);

\draw[arrows={-angle 45}, shorten >=6, shorten <=6]  (2, 0.4)--(0,0.4);
\draw[arrows={-angle 45}, shorten >=6, shorten <=6]  (1, 1.2)--(0,0.4);
\draw[arrows={-angle 45}, shorten >=6, shorten <=6]  (0,0.4)--(1.5, -0.8);
\draw[arrows={-angle 45}, shorten >=6, shorten <=6]  (0,0.4)--(0.5, -0.8);
\draw[arrows={-angle 45}, shorten >=6, shorten <=6]  (0.5, -0.8)--(1.5, -0.8);
\draw[arrows={-angle 45}, shorten >=6, shorten <=6]  (1, 1.2)--(0.5, -0.8);

\draw[arrows={-angle 45}, shorten >=6, shorten <=6] (0.5,-0.8)--(2,0.4);

    \node at(0,0.4) {$4$};
	\node at (2,0.4) {$5$};
	\node at (1,1.2) {$1$};
	\node at (0.5, -0.8) {$3$};
	\node at(1.5, -0.8) {$2$};
	
\end{tikzpicture}
\end{center}
\quad 
\begin{samepage}
${\rm Bl}_3\mathbb{C}P^2$:
\begin{center}
\begin{tikzpicture}[scale=2.0]

\draw[arrows={-angle 45}, shorten >=6, shorten <=6]  (0, 0)--(1.5,0);
\draw[arrows={-angle 45}, shorten >=6, shorten <=6]  (0, 0)--(0.75,-0.8);
\draw[arrows={-angle 45}, shorten >=6, shorten <=6]  (0.75,-0.8)-- (1.5, 0);
\draw[arrows={-angle 45}, shorten >=6, shorten <=6]  (0.75,-0.8)-- (1.5, 0.8);
\draw[arrows={-angle 45}, shorten >=6, shorten <=6]  (1.5, 0)--(1.5, 0.8);
\draw[arrows={-angle 45}, shorten >=6, shorten <=6]  (1.5, 0)--(0.75, 1.6);
\draw[arrows={-angle 45}, shorten >=6, shorten <=6]  (1.5, 0.8)--(0.75, 1.6);
\draw[arrows={-angle 45}, shorten >=6, shorten <=6]  (1.5, 0.8)--(0, 0.8);
\draw[arrows={-angle 45}, shorten >=6, shorten <=6]  (0.75, 1.6)--(0, 0.8);
\draw[arrows={-angle 45}, shorten >=6, shorten <=6]  (0.75, 1.6)--(0, 0);
\draw[arrows={-angle 45}, shorten >=6, shorten <=6]  (0, 0.8)--(0, 0);
\draw[arrows={-angle 45}, shorten >=6, shorten <=6]  (0, 0.8)--(0.75, -0.8);

\node at (0,0) {$4$};
\node at  (1.5, 0) {$5$};
\node at (0,0.8) {$6$};
\node at (1.5,0.8) {$3$};
\node at (0.75,1.6) {$1$};
\node at (0.75,-0.8) {$2$};
\end{tikzpicture}
\end{center}
\end{samepage}

For each of these quivers $Q(X)$, associated to a toric del Pezzo surface $X$, we choose a non-degenerate potential $S(X)$, which we assume to be sufficiently general in a sense made precise in the course of the proof of Lemma \ref{initial} below. For any sequence ${\bf i}=(i_1,\ldots,i_N)$ of vertices in $Q(X)$ without repetitions, we define
$$(Q_{\bf i}(X),S_{\bf i}(X))=\mu_{i_N}\cdots\mu_{i_1}(Q(X),S(X)).$$

Since the mutation operations on LG seeds, skew-symmetric matrices and quivers, respectively, are compatible, we have
$$Q(B({\bf s}_{\bf i}(X)))=Q_{\bf i}(X).$$

For example, for the quiver associated to $X=\mathbb{C}P^2$, the quivers $Q_{\bf i}(X)$ are (all) Markov quivers, that is, triangles with $3a$, $3b$, $3c$ parallel arrows, respectively, where $(a,b,c)$ is a Markov triple.\\[1ex]
We now claim:

\begin{lemma}\label{initial} For every toric del Pezzo surface $X$, and every sufficiently general potential $S(X)$ of $Q(X)$, there exists a virtual representation $P(X)$ of $(Q(X),S(X))$ of the form $P(X)=[V]-[(S_i^-)^c]$ for nonnegative $c$, such that
$${\rm CC}_{P(X)}(x)=\Phi_{{\bf s}(X)}(W(X)).$$
In other words, the LG potential $W(X)\in\mathbb{C}[z_1^\pm,z_2^\pm]$ associated to $X$ in Theorem \ref{delpezzoseeds} corresponds to the cluster character ${\rm CC}_{P(X)}(x)\in \mathbb{C}[x_1^\pm,\ldots,x_n^\pm]$ of $P(X)$ under the map $\Phi_{{\bf s}(X)}$.
\end{lemma}

%Note that the way they are depicted, for every of the above quivers there is a canonical {\it outer} cycle containing every vertex once.

\begin{proof} We fix a vertex $i_0$ of $Q$ arbitrarily. We then define a representation $V$ of the quiver with potential $(Q(X),S(X))$ as follows: we have $V_{i_0}$=0 and $V_i=\mathbb{C}$ for all $i\not=i_0$. Moreover, we have $V^-_i=0$ for all $i$. We have $V_\alpha=0$ if $\alpha:i\rightarrow j$ is an arrow bypassing $i_0$, that is, if there exist arrows $i\rightarrow i_0\rightarrow j$. Whenever such an arrow exists, the Jacobian relations require a certain linear combination of the maps representing the paths from $j$ to $i$ to be zero; we choose the potential sufficiently general such that this linear combination is sufficiently general. Recalling the maps $\gamma_i$ used in the definition of ${\bf g}$-vectors, we can then still choose the maps $V_\beta$ representing all other arrows (not incident with $i_0$) sufficiently general such that all $\gamma_i$ for $i\not=i_0$, which are not forced to be zero anyway, are of maximal possible rank. We moreover define $c=\dim{\rm Ker}(\gamma_{i_0})+1$.\\[1ex]
It is now just a matter of case-by-case direct calculations to verify that $P(X)=[V]-[(S_1^-)^c]$ fulfills the assumption of the lemma. We exemplify several of these calculations. We fix, for example, $i_0=1$.\\[1ex]
In the case $X=\mathbb{C}P^2$, the representation $V$ clearly fulfills the Jacobian relations since there are no consecutive arrows supporting non-zero maps in $V$. We have
$$y_1=x_2^{-3}x_3^{3},\; y_2=x_1^{3}x_3^{-3},\; y_3=x_1^{-3}x_2^{3},$$
and thus
$$F_P(y)=F_V(y)=1+y_2+y_2y_3=1+x_1^{3}x_3^{-3}+x_2^3x_3^{-3}.$$
The $\gamma$-maps are of the form $0=\gamma_2:\mathbb{C}^3\rightarrow 0$ and $0=\gamma_3:0\rightarrow\mathbb{C}^3$, so that $${\bf g}_V=(\dim{\rm Ker}(\gamma_1),2,-1),$$ and thus $${\bf g}_P=(-1,2,-1),$$ yielding
$${\rm CC}_P(x)=x_1^{-1}x_2^2x_3^{-1}(1+x_1^3x_3^{-3}+x_2^3x_3^{-3})=x_1^{-1}x_2^2x_3^{-1}+x_1^{-1}x_2^{-1}x_3^2+x_1^2x_2^{-1}x_3^{-1}.$$
This indeed coincides with $\Phi(z_1+z_2+z_1^{-1}z_2^{-1})$ since, by definition,
$$\Phi(z_1)=x_1^{-1}x_2^2x_3^{-1},\; \Phi(z_2)=x_1^{-1}x_2^{-1}x_3^2.$$

The verification in the cases  $X=\mathbb{C}P^1\times\mathbb{C}P^1$ and $X={\rm Bl}_1\mathbb{C}P^2$ is entirely similar; we explain how to derive the cluster characters. For $X=\mathbb{C}P^1\times\mathbb{C}P^1$, we have $$y_1=x_2^2x_3^{-2},\; y_2=x_1^{-2}x_4^2,\; y_3=x_1^2x_4^{-2},\; y_4=x_2^{-2}x_3^2$$ and $F_V(y)=1+y_3+y_3y_4+y_2y_3y_4$. We find $\mathbf{g}_P=(-1,1,-1,1)$, and thus $${\rm CC}_P(x)=x_1^{-1}x_2x_3^{-1}x_4+x_1x_2x_3^{-1}x_4^{-1}+x_1x_2^{-1}x_3x_4^{-1}+x_1^{-1}x_2^{-1}x_3x_4.$$
For $X={\rm Bl}_1\mathbb{C}P^2$, we have
$$y_1=x_2^{-3}x_3x_4^2,\; y_2=x_1^3x_3^{-2}x_4^{-1},\; y_3=x_1^{-1}x_2^2x_4^{-1},\; y_4=x_1^{-2}x_2x_3$$
and $F_V(y)=1+y_2+y_2y_3+y_2y_3y_4$. We find $\mathbf{g}_P=(-1,-1,1,1)$ and thus
$${\rm CC}_P(x)=x_1^{-1}x_2^{-1}x_3x_4+x_1^2x_2^{-1}x_3^{-1}+x_1x_2x_3^{-1}x_4^{-1}+x_1^{-1}x_2^{2}x_4^{-1}.$$
For the remaining two cases, we first depict representations of the relevant quivers:

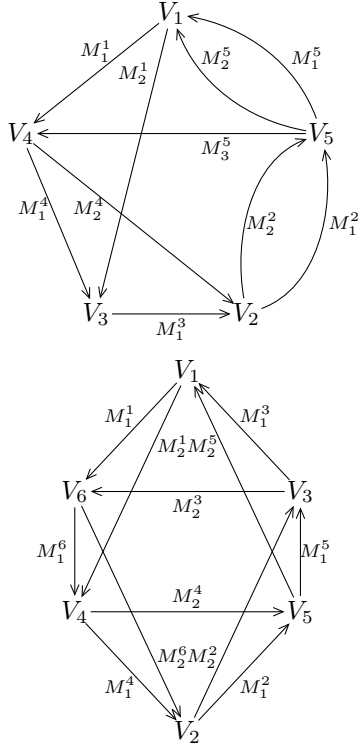
\begin{center}
\begin{tikzpicture}[scale=2.0]
\draw[arrows={-angle 45}, shorten >=6, shorten <=6]  (2,0.4) to [out=170, in=290] (1,1.2);
\draw[arrows={-angle 45}, shorten >=6, shorten <=6]  (2,0.4) to [out=110, in=350] (1,1.2);
\draw[arrows={-angle 45}, shorten >=6, shorten <=6]  (1.5, -0.8) to [out=20, in=280] (2,0.4);
\draw[arrows={-angle 45}, shorten >=6, shorten <=6]  (1.5, -0.8)  -- (2,0.4);

\draw[arrows={-angle 45}, shorten >=6, shorten <=6]  (2, 0.4)--(0,0.4);
\draw[arrows={-angle 45}, shorten >=6, shorten <=6]  (1, 1.2)--(0,0.4);
\draw[arrows={-angle 45}, shorten >=6, shorten <=6]  (0,0.4)--(1.5, -0.8);
\draw[arrows={-angle 45}, shorten >=6, shorten <=6]  (0,0.4)--(0.5, -0.8);
\draw[arrows={-angle 45}, shorten >=6, shorten <=6]  (0.5, -0.8)--(1.5, -0.8);
\draw[arrows={-angle 45}, shorten >=6, shorten <=6]  (1, 1.2)--(0.5, -0.8);

\draw[arrows={-angle 45}, shorten >=6, shorten <=6] (0.5,-0.8)--(2,0.4);

    \node at(0,0.4) {$V_4$};
	\node at (2,0.4) {$V_5$};
	\node at (1,1.2) {$V_1$};
	\node at (0.5, -0.8) {$V_3$};
	\node at(1.5, -0.8) {$V_2$};
	\node at (0.5, 0.95){$\scriptstyle{M^1_1}$};
	\node at (0.75, 0.8){$\scriptstyle{M^1_2}$};
	\node at (0.085, -0.1){$\scriptstyle{M^4_1}$};
	\node at (0.45, -0.1){$\scriptstyle{M^4_2}$};
	\node at (1, -0.9){$\scriptstyle{M^3_1}$};
	\node at (2.15, -0.2){$\scriptstyle{M^2_1}$};
	\node at (1.6, -0.2){$\scriptstyle{M^2_2}$};
	\node at (1.3, 0.9){$\scriptstyle{M^5_2}$};
	\node at (1.3, 0.3){$\scriptstyle{M^5_3}$};
	\node at (1.9, 0.9){$\scriptstyle{M^5_1}$};
	\node at (1.2,-0.1){$\scriptstyle{M^3_2}$};
	
\end{tikzpicture}

\begin{tikzpicture}[scale=2.0]

\draw[arrows={-angle 45}, shorten >=6, shorten <=6]  (0, 0)--(1.5,0);
\draw[arrows={-angle 45}, shorten >=6, shorten <=6]  (0, 0)--(0.75,-0.8);
\draw[arrows={-angle 45}, shorten >=6, shorten <=6]  (0.75,-0.8)-- (1.5, 0);
\draw[arrows={-angle 45}, shorten >=6, shorten <=6]  (0.75,-0.8)-- (1.5, 0.8);
\draw[arrows={-angle 45}, shorten >=6, shorten <=6]  (1.5, 0)--(1.5, 0.8);
\draw[arrows={-angle 45}, shorten >=6, shorten <=6]  (1.5, 0)--(0.75, 1.6);
\draw[arrows={-angle 45}, shorten >=6, shorten <=6]  (1.5, 0.8)--(0.75, 1.6);
\draw[arrows={-angle 45}, shorten >=6, shorten <=6]  (1.5, 0.8)--(0, 0.8);
\draw[arrows={-angle 45}, shorten >=6, shorten <=6]  (0.75, 1.6)--(0, 0.8);
\draw[arrows={-angle 45}, shorten >=6, shorten <=6]  (0.75, 1.6)--(0, 0);
\draw[arrows={-angle 45}, shorten >=6, shorten <=6]  (0, 0.8)--(0, 0);
\draw[arrows={-angle 45}, shorten >=6, shorten <=6]  (0, 0.8)--(0.75, -0.8);

\node at (0,0) {$V_4$};
\node at (1.5, 0) {$V_5$};
\node at (0,0.8) {$V_6$};
\node at (1.5,0.8) {$V_3$};
\node at (0.75,1.6) {$V_1$};
\node at (0.75,-0.8) {$V_2$};
\node at (0.3, 1.3){$\scriptstyle{M^1_1}$};
\node at (0.65, 1.1){$\scriptstyle{M^1_2}$};
\node at (1.2, 1.3){$\scriptstyle{M^3_1}$};
\node at (0.85, 1.1){$\scriptstyle{M^5_2}$};

\node at (0.3, -0.5){$\scriptstyle{M^4_1}$};
\node at (0.65, -0.3){$\scriptstyle{M^6_2}$};
\node at (1.2, -0.5){$\scriptstyle{M^2_1}$};
\node at (0.85, -0.3){$\scriptstyle{M^2_2}$};

\node at (-0.15, 0.4){$\scriptstyle{M^6_1}$};
\node at (1.6, 0.4){$\scriptstyle{M^5_1}$};

\node at (0.75, 0.7){$\scriptstyle{M^3_2}$};
\node at (0.75, 0.1){$\scriptstyle{M^4_2}$};

\end{tikzpicture}
\end{center}

In the case $X={\rm Bl}_2\mathbb{C}P^2$, the scalar $M^5_3$ is zero by definition, and a sufficiently general linear combination of the scalars representing the paths from $4$ to $5$ is zero to satisfy the Jacobian relations. The maps $\gamma_2:V_5^2\rightarrow V_3\oplus V_4$ and $\gamma_3:V_2\oplus V_5\rightarrow V_1\oplus V_4$ are automatically zero since all paths between the respective vertices pass through vertex $1$ or through the arrow from $5$ to $4$. All non-zero scalars representing the arrows can then be chosen to be sufficiently general so that the maps $$\gamma_4:\mathbb{C}^2\simeq V_2\oplus V_3\rightarrow V_1\oplus V_5\simeq\mathbb{C}\mbox{ and }\gamma_5:\mathbb{C}\simeq V_1^2\oplus V_4\rightarrow V_3\oplus V_2^2\simeq\mathbb{C}^3$$ are of maximal rank.
%  there are two new phenomena to be discussed. First of all, {\it we assume that the potential $S$ does not involve the outer $5$-cycle with non-zero coefficient} (otherwise $V$ had to fulfill a Jacobian relation which might force the compositions $M_1^2M_1^3M_1^4$ and $M_2^2M_1^3M_1^4$ to be proportional). Moreover, in computing the ${\bf g}$-vector, we will see why it is necessary to represent the arrows outside the outer cycle by zero. Namely, this forces the $\gamma$-maps $\gamma_2:\mathbb{C}^2\rightarrow\mathbb{C}^2$ and $\gamma_3:\mathbb{C}^2\rightarrow\mathbb{C}$ to be zero, whereas $\gamma_4:\mathbb{C}^2\rightarrow\mathbb{C}$ is non-zero since it involves a linear combination of $M_1^2$ and $M_2^2$ which we can, in general, assume to be non-zero.
The resulting ${\bf g}$-vector is thus
$${\bf g}_V=(\dim{\rm Ker}(\gamma_1),1,1,0,-1),$$
and thus
$${\bf g}_P=(-1,1,1,0,-1).$$
The rest of the verification is again straight-forward, so that both ${\rm CC}_P(x)$ and $\Phi(W(x))$ coincide with
$$x_1^{-1}x_2x_3x_5^{-1}+x_1x_2^{-1}x_4x_5^{-1}+x_1x_2^{-1}x_3^{-1}x_5+x_1^{-1}x_2x_4^{-1}x_5+x_3^{-1}x_4^{-1}x_5^2.$$

The verification in the final case $X={\rm Bl}_3\mathbb{C}P^2$ is again similar, as a direct inspection of the representation shows. This finishes the proof of the lemma.\end{proof}

We can now mutate the virtual representation $P(X)$ to all mutated quivers with potentials:

\begin{definition} For every sequence of vertices ${\bf i}=(i_1,\ldots,i_N)$ without repetition, define $P_{\bf i}(X)$ as the representation of $(Q_{\bf i}(X),S_{\bf i}(X))$ given by
$$P_{\bf i}(X)=\mu_{i_N}\cdots\mu_{i_1}P(X).$$
\end{definition}

The main result of this paper now follows immediately:

\begin{theorem}\label{main} Every Floer potential of a monotone Lagrangian torus $L_{\bf i}(X)$ appearing in a Lagrangian seed with associated LG seed $s_{\bf i}(X)$ %$W_{\bf i}(X)$ of an LG seed ${\bf s}_{\bf i}(X)$
 corresponds to the cluster character ${\rm CC}_{P_{\bf i}(X)}(x)$ of the virtual representation $P_{\bf i}(X)$ of $(Q_{\bf i}(X),S_{\bf i}(X))$ under the comparison map $\Phi_{{\bf s}_{\bf i}(X)}$, that is,
$$\Phi_{{\bf s}_{\bf i}(X)}(W_{L_{\bf i}(X)})={\rm CC}_{P_{\bf i}(X)}(x).$$
\end{theorem}

Indeed, the claimed equality holds for the empty sequence ${\bf i}$, the right hand side is mutation invariant by Lemma \ref{wk} and the definition of the $P_{\bf i}(X)$, and the left hand side is mutation invariant by Lemma \ref{phimap}, the definition of the $W_{\bf i}(X)$, and the mutation compatibility of Floer potentials of Theorem \ref{mfp}.\\[1ex]
{\bf Remark:} The authors expect that, using \cite{TV}, explicit (graded) potentials for the above five quivers $Q(X)$ can be constructed, so that the representations $P_{\bf i}(X)$, and conjecturally also their cluster characters, can be determined more explicitly as graded representations of the $Q_{\bf i}(X)$.

\end{document}